\documentclass[12pt]{article}
\usepackage{libertine}
\usepackage{geometry}

\usepackage{amsthm}
\theoremstyle{plain}
\usepackage{amsfonts}
\usepackage{amssymb}
\usepackage{amsmath}
\usepackage{hyperref}
\hypersetup{
    colorlinks=true,
    linkcolor=teal,
    filecolor=teal,
    urlcolor=teal,
    citecolor=teal,
}
\usepackage{cleveref}
\usepackage[all]{xy}
\usepackage{tikz-cd}
\usepackage{etoolbox}
\apptocmd{\sloppy}{\hbadness 10000\relax}{}{}

\makeatletter
\let\@fnsymbol\@arabic
\makeatother

\newtheorem*{theorem*}{Theorem}
\newtheorem{proposition}{Proposition}
\newtheorem*{proposition*}{Proposition}
\newtheorem{definition}{Definition}
\newtheorem*{definition*}{Definition}
\newtheorem{remark}{Remark}
\newtheorem*{remark*}{Remark}

\newtheorem*{lemma*}{Lemma}

\newtheorem*{corollary*}{Corollary}

\DeclareMathOperator{\ob}{Ob}

\DeclareMathOperator{\id}{Id}

\newcommand{\set}[1]{{\{\,{#1}\,\}}}
\newcommand{\range}[2]{{\{{#1}, \dots,{#2}\}}}

\newcommand{\R}{{\mathbb{R}}}
\newcommand{\Cat}{{\mathcal{C}}}
\newcommand{\meas}{{\mathcal{M}}}
\newcommand{\func}{{\mathcal{F}}}
\newcommand{\Vect}{{\mathbf{Vect}}}
\newcommand{\Mod}{{\mathbf{Mod}}}
\newcommand{\CAlg}{{\mathbf{CAlg}}}
\newcommand{\Subm}{{\mathbf{Subm}}}
\newcommand{\Meas}{\mathbf{Meas}_{0\textnormal{R}}}
\newcommand{\Gr}{{\mathbf{Gr}}}
\newcommand{\CAT}{{\mathbf{Cat}}}
\newcommand{\superscript}[1]{^\textnormal{#1}}
\newcommand{\density}{|\Lambda|}
\newcommand{\sections}[1]{C^\infty({#1})}
\newcommand{\sectionscs}[1]{C^\infty_0({#1})}
\newcommand{\anon}{{\text{\textendash}}}

\newcommand{\extranat}{\overset{\cdot\cdot}{\Rightarrow}}

\newcounter{diagram}

\newenvironment{diagram}{\setcounter{diagram}{\value{equation}}\refstepcounter{diagram}}{}

\crefname{diagram}{diag.}{diags.}
\Crefname{diagram}{Diagram}{Diagrams}
\creflabelformat{diagram}{#2(#1)#3}

\title{Neural network layers as parametric spans}
\author{
    Mattia G. Bergomi
    \thanks{Correspondence at \url{mattiagbergomi@gmail.com}}
    \and Pietro Vertechi
    \thanks{Correspondence at \url{pietro.vertechi@protonmail.com}}
}
\date{}
\begin{document}
\maketitle
\begin{abstract}
    Properties such as composability and automatic differentiation made artificial neural networks a pervasive tool in applications. Tackling more challenging problems caused neural networks to progressively become more complex and thus difficult to define from a mathematical perspective. We present a general definition of linear layer arising from a categorical framework based on the notions of {\em integration theory} and {\em parametric spans}. This definition generalizes and encompasses classical layers (e.g., dense, convolutional), while guaranteeing existence and computability of the layer's derivatives for backpropagation.
\end{abstract}

\section{Introduction}

In recent years, {\em artificial neural networks} have been applied to ever more general problems, incorporating the most diverse operators and intricate architectures. Unlike the initial definitions~\cite{rumelhart1986learning}, which could be easily formalized as directed graphs, modern neural networks do not obey a precise mathematical definition.

We use the general language of category theory to define a broad class of linear layer structures, which encompasses most classically known examples---dense and convolutional layers, as well as geometric deep learning layers.
The key ingredient is a general, categorical definition of {\em integration theory} that, combined with the notion of {\em parametric spans}, yields a flexible framework where layer-like bilinear operators can be studied.

For machine learning applications, not only the \textit{activation values} of a model are important, but also its derivatives with respect to the parameters. Reverse-mode automatic differentiation~\cite{baydin2018automatic} is a modern, popular technique to address this issue. It attempts to define rules to backpropagate dual vectors of the output to dual vectors of the parameters or of the input. The existence of such rules is a guiding principle for our framework: we will show that, for parametric span-based layers, the reverse-mode differentiation rule can be obtained by permuting the legs of the span.

\paragraph{Structure.} In \cref{sec:integration_theories}, we introduce the notion of {\em Frobenius integration theory}, which generalizes Lebesgue integration to arbitrary source categories. The primary example we discuss is the category of manifolds and submersions. In \cref{sec:parametric_spans}, we use the notion of integration theory in tandem with parametric spans to define bilinear operators with a straightforward reverse-mode differentiation rule. This will allow us to recover several well-known linear neural network layers in \cref{sec:classical_architectures}.

\section{Integration theories}
\label{sec:integration_theories}

Our goal is to represent the {\em structure} of a linear layer of a neural network---a bilinear map from the input and parameters to the output---via a collection of maps in a familiar category. We aim to build a simple framework that is sufficiently flexible to cover most popular linear neural network layers and allow for novel generalizations. As backpropagation is crucial for deep learning, we also require that the dual of the linear layers we define can be computed effectively in our framework. In this section, we establish the necessary preliminary notions to achieve that.

Let $\Vect_K$ and $\CAlg_K$ denote the categories of vector spaces and commutative algebras over a base field $K$. For a commutative $K$-algebra $A$, let $\Mod_A$ denote the category of modules over $A$, and let $\Mod_A / K$ denote the comma category~\cite[Sect.~II.6]{mac2013categories} of
\begin{equation*}
    \Mod_A \rightarrow \Vect_K \xleftarrow{K} \bullet.
\end{equation*}
Explicitly, an object in $\Mod_A / K$ is an $A$-module $M$ equipped with a $K$-linear functional $\epsilon \colon M \rightarrow K$. A morphism $(M_1, \epsilon_1)\rightarrow (M_2, \epsilon_2)$ in $\Mod_A / K$ is an $A$-module homomorphism that makes the following diagram commute.
\begin{equation*}
    \begin{tikzcd}[column sep=small]
        M_1 \arrow[swap]{rd}{\epsilon_1} \arrow[dashed]{rr} && M_2 \arrow{ld}{\epsilon_2} \\
        & K &
    \end{tikzcd}
\end{equation*}
Finally, we denote $\Gr(\Mod/K)$ the {\em covariant Grothendieck construction}~\cite[Sect.~5.5]{richter2020categories} associated to the functor $\Mod / K \colon \CAlg_K\superscript{op} \rightarrow \CAT$.

\begin{remark}
    The category $\Gr(\Mod/K)$ can also be denoted $\mathbf{Lens}_{\Mod/K}\superscript{op}$, see~\cite[Def.~3.3]{spivak2019generalized}.
\end{remark}

\begin{remark}
    For simplicity, we have chosen to work with modules over commutative $K$-algebras. However, the definitions work more generally. Our main results---\cref{prop:bilinear_operator,prop:operator_adjoint}---can be proved diagrammatically and hold for modules over commutative algebras in any symmetric monoidal category. In particular, to introduce a notion of continuity, one could work with normed modules over normed commutative algebras.
\end{remark}

\begin{definition}\label{def:frobenius_integration_theory}
    A {\em Frobenius integration theory} on a category $\Cat$ consists of a functor
    \begin{equation*}
        \Cat\rightarrow \Gr(\Mod/K).
    \end{equation*}
\end{definition}

It is helpful to unpack \cref{def:frobenius_integration_theory}. A functor $\Cat\rightarrow \Gr(\Mod/K)$ associates to each object $X\in\ob(\Cat)$ a commutative algebra $\func(X)$ and a $\func(X)$-module $\meas(X)$, as well as a $K$-linear functional $\int_X\colon \meas(X)\rightarrow K$.

\paragraph{Functors.} $\func$ and $\meas$ individually can be regarded as functors
\begin{equation*}
    \func \colon\Cat\superscript{op}\rightarrow \CAlg_K
    \text{ and }
    \meas \colon \Cat \rightarrow \Vect_K.
\end{equation*}
For simplicity, given a morphism $f\colon X\rightarrow Y$ in $\Cat$, we will use the pullback and pushforward notation to refer to $\func(f)$ and $\meas(f)$:
\begin{equation*}
    f^* := \func(f)
    \quad \text{ and } \quad
    f_* := \meas(f).
\end{equation*}

\paragraph{Action.} We denote $\cdot$ the action of $\func(X)$ on $\meas(X)$. Linearity under restriction of scalars corresponds to a condition akin to {\em Frobenius reciprocity} for adjoint functors. More explicitly, given a morphism $f\colon X\rightarrow Y$ in $\Cat$, $y\in \func(Y)$, and $\mu \in \meas(X)$, we have
\begin{equation}\label{eq:frobenius}
    f_*(f^* y \cdot \mu) = y\cdot f_*\mu.
\end{equation}
Note that we use the letter $y$ to denote an element of $\func(Y)$ and not a point of $Y$.

\paragraph{Functional.} Finally, $\int$ is a family of $K$-linear functionals $\int_X\colon \meas(X)\rightarrow K$ such that
\begin{equation}\label{eq:naturality}
    \int_X \mu = \int_Y f_*\mu,
\end{equation}
for all $\mu \in \meas(X)$ and $f\colon X \rightarrow Y$.

\begin{proposition}\label{prop:frobenius_integration_theory}
    The data of a Frobenius integration theory on an arbitrary source category $\Cat$ is equivalent to the following.
    \begin{itemize}
        \item Functors $\func \colon\Cat\superscript{op}\rightarrow \CAlg_K$ and $\meas \colon\Cat\rightarrow \Vect_K$.
        \item A family of actions $\func(X) \otimes \meas(X) \rightarrow \meas(X)$ that respects \cref{eq:frobenius}.
        \item A family of $K$-linear functionals $\int_X\colon \meas(X)\rightarrow K$ that respects \cref{eq:naturality}.
    \end{itemize}
\end{proposition}
\begin{proof}
    Straightforward verification.
\end{proof}

\begin{proposition}\label{prop:extranaturality}
    $\int$ and $\cdot$ induce an extranatural transformation~\cite{eilenberg1966generalization} (denoted by $\extranat$)
    \begin{equation*}
        \func \otimes \meas \extranat K,
        \quad\text{ given by } x \otimes \mu \mapsto \int_X x\cdot \mu.
    \end{equation*}
    In other words, for all $f\colon X \rightarrow Y$, $\mu \in \meas(X)$, and $y \in \func(Y)$,
    \begin{equation}\label{eq:extranaturality}
        \int_X f^* y \cdot \mu = \int_Y y\cdot f_*\mu.
    \end{equation}
\end{proposition}
\begin{proof}
    By direct computation,
    \begin{alignat*}{4}
        \int_X f^* y \cdot \mu
        &= \int_Y f_*(f^* y \cdot \mu) && \quad\quad\text{ by \cref{eq:naturality}}\\
        &= \int_Y y\cdot f_*\mu && \quad\quad\text{ by \cref{eq:frobenius}}.
    \end{alignat*}
\end{proof}

\subsection{Examples}
\label{sec:examples}

\subsubsection*{Measurable spaces}

To form an intuition on \cref{def:frobenius_integration_theory}, it is helpful to think of $\func(X)$ as {\em functions} over $X$, and of $\meas(X)$ as {\em measures} over $X$. Indeed, an important example of integration theory comes from the category of measurable spaces $\Meas$, whose objects are measurable spaces equipped with a $\sigma$-ideal of measure $0$ subsets, and whose morphisms are equivalence classes of nullset-reflecting measurable functions.

To each $X\in\ob(\Meas)$, we associate the algebra $L^\infty(X)$ of equivalence classes of essentially bounded measurable function and the $L^\infty(X)$-module $ba(X)$ of bounded and finitely additive signed measures. To show that this is indeed a Frobenius integration theory, we need to verify \cref{eq:frobenius,eq:naturality}.
\Cref{eq:frobenius} follows from the adjunction between pullback of a function and pushforward of a measure. More explicitly, given a nullset-reflecting measurable function $f\colon X\rightarrow Y$, a measurable subspace $S\subseteq Y$, $y \in L^\infty(Y)$, and $\mu \in ba(X)$,
\begin{equation*}
    (y\cdot f_*\mu)(S) = \int_S y\, d(f_*\mu) = \int_{f^{-1}(S)} f^* y\, d\mu = f_*(f^*y\cdot \mu)(S),
\end{equation*}
hence the measures $y\cdot f_*\mu$ and $f_*(f^*y\cdot \mu)$ coincide. Verifying \cref{eq:naturality} is straightforward:
\begin{equation*}
    \int_Y d(f_*\mu) = (f_*\mu)(Y) = \mu(f^{-1}(Y)) = \mu(X) = \int_X d\mu.
\end{equation*}

\subsubsection*{Submersions}

The aim of this section is to define an {\em integration theory} based on smooth spaces, which we will use to give practical examples of neural network layers. To proceed, we will need a few technical assumptions. Whenever we use the word manifold, we refer to smooth manifolds. Furthermore, we require manifolds to be paracompact Hausdorff spaces. We remind the reader that a submersion is a smooth map whose differential is, at every point, surjective. We denote $\Subm$ the category of manifolds and submersions.

We can associate to a manifold its space of smooth real-valued functions $\sections{X}$.
This extends to a functor $\Subm\superscript{op} \rightarrow \CAlg_\R$ via pullback of functions (precomposition). Given a submersion $f\colon X \rightarrow Y$ and $y \in \sections{Y}$, we denote the pullback $f^*y$.
We denote $\sectionscs{\density_X}$ the space of smooth densities~\cite[Sect.~1.1]{berline2003heat} of compact support. Given a submersion $f\colon X \rightarrow Y$ and a density of compact support $\mu \in \sectionscs{\density{_X}}$, we denote the pushforward $f_*\mu$. Thus, $\sectionscs{\density_X}$ extends to a covariant functor $\Subm\rightarrow\Vect_\R$.
\begin{remark}
    The pushforward of a smooth density is well defined and smooth for {\em proper} submersions. However, here we are working with densities of compact support, so $f$ is automatically proper on the support of $\mu$.
\end{remark}

The pointwise multiplication map
\begin{equation*}
    \sections{X} \otimes \sectionscs{\density_{X}}
    \rightarrow \sectionscs{\density_{X}},
\end{equation*}
endows $\sectionscs{\density_{X}}$ with the structure of a $\sections{X}$-module. To prove \cref{eq:frobenius}, let us fix a positive density $\nu$ on $Y$. Then, for all point $p\in Y$,
\begin{align*}
    f_*(f^*y\cdot \mu)(p)
    &= \nu(p)\int_{f^{-1}(p)} (f^*y\cdot \mu)/f^*(\nu)\\
    &= \nu(p)\int_{f^{-1}(p)} f^*y\cdot (\mu/f^*(\nu))\\
    &= y(p)\nu(p)\int_{f^{-1}(p)} (\mu/f^*(\nu))\\
    &= (y \cdot f_*\mu)(p).
\end{align*}
Verifying \cref{eq:naturality} is also straightforward. Hence, we can conclude that
\begin{equation*}
    \sections{\anon},\, \sectionscs{\density{_\anon}},\, \int
\end{equation*}
define a Frobenius integration theory. In \cref{sec:classical_architectures}, we will use this particular integration theory to recover several classical neural network layers.

\section{Parametric spans}
\label{sec:parametric_spans}

To formally describe neural network layers with locality and weight sharing constraints, we introduce the notion of {\em parametric span}---a span with an added space of parameters (or weights). It is represented by the following diagram.
\begin{diagram}\label{diagram:parametric_span}
    \begin{equation}
        \begin{tikzcd}
            & E \arrow[swap]{ld}{s} \arrow{d}{\pi} \arrow{rd}{t} & \\
            X & W & Y
        \end{tikzcd}
    \end{equation}
\end{diagram}
In this representation, $X$ represents the space of {\em input data}, $Y$ the space of {\em output data}, $E$ the space of {\em edges}, and $W$ the space of {\em weights}.
Intuitively, this is an abstract representation of the notions of {\em locality} and {\em weight sharing} in deep learning. The span
\begin{equation*}
    \begin{tikzcd}
        & E \arrow[swap]{ld}{s} \arrow{rd}{t} & \\
        X & & Y
    \end{tikzcd}
\end{equation*}
determines the connectivity structure of the network (which inputs are connected to which outputs).
The map
\begin{equation*}
    \begin{tikzcd}
        E  \arrow{d}{\pi} \\
        W
    \end{tikzcd}
\end{equation*}
enforces weight sharing along the fibers of $\pi$.

\begin{proposition}\label{prop:bilinear_operator}
    A parametric span, as in \cref{diagram:parametric_span}, induces a $K$-linear map
    \begin{equation}\label{eq:bilinear_operator}
        \begin{aligned}
            \func(X) \otimes \func(W) \otimes \meas(E) &\rightarrow \meas(Y)\\
            x \otimes w \otimes \mu &\mapsto t_* (s^* x \cdot \pi^* w\cdot \mu).
        \end{aligned}
    \end{equation}
\end{proposition}
\begin{proof}
    The above map can be obtained as
    \begin{equation*}
        \func(X) \otimes \func(W) \otimes \meas(E) \xrightarrow{s^* \, \otimes \, \pi^* \, \otimes \, \id}
        \func(E) \otimes \func(E) \otimes \meas(E) \rightarrow
        \meas(E) \xrightarrow{t_*}
        \meas(Y).
    \end{equation*}
\end{proof}
\noindent In practical applications, we will fix $\mu \in \meas(E)$ and consider the map in \cref{eq:bilinear_operator} as a bilinear map from the input and parameters to the output.
\begin{proposition}\label{prop:operator_adjoint}
    For all parametric span as in \cref{diagram:parametric_span}, the following diagram commutes.
    \begin{equation*}
        \begin{tikzcd}
            \func(X) \otimes \func(Y) \otimes \func(W) \otimes \meas(E)
            \arrow{r}{}
            \arrow{d}{}
            & \func(X) \otimes \meas(X) \arrow{d}{}\\
            \func(Y) \otimes \meas(Y) \arrow{r}{} & K
        \end{tikzcd}
    \end{equation*}
    Equivalently, in formulas, for all $x\in \func(X), \, y\in \func(Y), \, w\in \func(W)$, and $\mu\in\meas(E)$,
    \begin{equation}\label{eq:operator_adjoint}
        \int_Y y \cdot t_* (s^* x \cdot \pi^* w \cdot \mu) = \int_X x \cdot s_* (t^* y \cdot \pi^* w \cdot \mu).
    \end{equation}
\end{proposition}
\begin{proof}
    \Cref{eq:operator_adjoint} can be proved via direct calculation:
    \begin{align*}
        \int_Y y \cdot t_* (s^* x \cdot \pi^* w \cdot \mu)
        &= \int_E t^*y \cdot s^* x \cdot \pi^* w \cdot \mu &&\text{ by \cref{eq:extranaturality}}\\
        &= \int_E s^* x \cdot t^*y \cdot \pi^* w \cdot \mu &&\text{ by commutativity}\\
        &= \int_X x \cdot s_* (t^*y \cdot \pi^* w \cdot \mu) &&\text{ by \cref{eq:extranaturality}}.
    \end{align*}
\end{proof}

\Cref{prop:operator_adjoint} is especially relevant for reverse-mode differentiation, i.e., mapping a dual vector of the output to the corresponding dual vector of the input.
If the dual vector of the output is of the form
\begin{equation*}
    \int_Y y \cdot \anon
\end{equation*}
for some $y \in \func(Y)$, then thanks to \cref{eq:operator_adjoint}, the reverse-mode differentiation rule with respect to the input is
\begin{equation*}
    \int_Y y \cdot t_* (s^* \anon \cdot \pi^* w \cdot \mu) = \int_X \anon \cdot s_* (t^* y \cdot \pi^* w \cdot \mu).
\end{equation*}
Thus, the dual vector represented by $y \in \func(Y)$ is mapped to the dual vector represented by $s_* (t^* y \cdot \pi^* w \cdot \mu)\in \meas(X)$. In other words, the reverse-mode differentiation rule for the input---and, by symmetry, for the parameters---can be obtained by reordering the legs of the parametric span.

\section{Classical architectures}
\label{sec:classical_architectures}

Our framework encompasses radically different classical neural architectures. Roughly speaking, we will discuss discrete and continuous architectures, with or without symmetry (weight sharing).

\paragraph{Dense layer.} Multi-Layer Perceptrons (MLPs)~\cite{rumelhart1986learning} are the simplest neural network, as they are a discrete architecture with no symmetry based on matrix multiplication. In this non-equivariant case, i.e., when the network does not respect any symmetries of the problem, the map $\pi$ is an isomorphism. To see this in practice, let us consider a layer with $n_i$ input nodes and $n_o$ output nodes. For ease of notation, we identify each natural number element $n$ with the set $\range{0}{n-1}$. We define a discrete parametric span as follows.
\begin{diagram}\label{eq:MLP}
    \begin{equation}
        \begin{tikzcd}
            & n_i \times n_o \arrow[swap]{ld} \arrow{d} \arrow{rd} & \\
            n_i & n_i\times n_o & n_o
        \end{tikzcd}
    \end{equation}
\end{diagram}
Source and target maps are given by the product projections.

\paragraph{Convolutional layer.} Convolutional Neural Networks (CNNs)~\cite{lecun1989backpropagation} represent a more interesting case, as they introduce spatial symmetry and locality. Let us, for simplicity, consider a purely convolutional layer with $n_i$ input channels and $n_o$ output channels. Let $S_i, S_o$ denote the shapes of the input and output images, and let $F$ denote the shape of the filter. To define the parametric span, we can proceed as in the dense layer case, with an important difference: the map $\pi$ is no longer trivial, and fibers along $\pi$ represent output image shapes.
\begin{diagram}\label{eq:conv}
    \begin{equation}
        \begin{tikzcd}
            & n_i \times n_o \times F \times S_o \arrow[swap]{ld} \arrow{d} \arrow{rd} & \\
            n_i \times S_i & n_i \times n_o \times F & n_o \times S_o
        \end{tikzcd}
    \end{equation}
\end{diagram}
The target morphism and the weight sharing morphism are projections, whereas the source morphism relies on a linear map
\begin{equation*}
    F \times S_o \rightarrow S_i.
\end{equation*}
The coefficient of the first argument encodes the {\em dilation} of the convolutional layer, whereas the coefficient of the second argument encodes the {\em stride}.

\paragraph{Geometric deep learning.}
Neural networks for non-Euclidean domains, such as graphs or manifolds, share many features with CNNs and can be handled in a similar way. The formulation in~\cite{monti2017geometric} is particularly suitable to our framework for two reasons. On the one hand, it encompasses many other approaches (Geodesic CNN~\cite{masci2015geodesic}, Anisotropic CNN~\cite{boscaini2016learning}, Diffusion CNN~\cite{atwood2015diffusion}, Graph CN~\cite{kipf2016semi}). On the other hand, it can be directly translated into our formalism. The authors of~\cite{monti2017geometric} postulate a {\em neighborhood relation} $q \in \mathcal{N}(p)$ on a Riemannian manifold $X$, together with local $d$-dimensional coordinates $\mathbf{u}(p, q)$ on pairs of neighbors. In our framework, this translates to the following parametric span.
\begin{diagram}\label{eq:geometric}
    \begin{equation}
        \begin{tikzcd}
            & \set{(p, q) \mid q \in \mathcal{N}(p)} \arrow[swap]{ld} \arrow{d}{\mathbf{u}} \arrow{rd} & \\
            X & \R^d & X
        \end{tikzcd}
    \end{equation}
\end{diagram}
The source and target maps are projections, whereas the weight sharing map is given by the local coordinates $\mathbf{u}$. The Riemannian structure, on which geometric deep learning is based, naturally induces a density $\mu$ on $E = \set{(p, q) \mid q \in \mathcal{N}(p)}$.
\begin{remark}
    Some care is needed as $\mu$ will in general not have compact support: densities induced by a Riemannian metric are positive. Hence, if the space $E$ is not compact, it will be necessary to multiply $\mu$ with an appropriate bump function.
\end{remark}

\section{Discussion}

We provide categorical foundations to the study of linear layers in deep learning. We abstract away the key ingredients to define linear layers (i.e., bilinear maps) and describe them in categorical terms. Our framework is based on two pillars: integration theories and parametric spans. Both notions are valid in arbitrary source categories, thus granting full generality to our approach.

Not only computing values (the forward pass) but also computing derivatives (the backward pass) is crucial in deep learning. Guided by this principle, we devise our framework in such a way that the backward pass has the same structure as the forward pass and, therefore, a comparable computational cost.

To examine concrete examples, we primarily explore integration theories on the category of nullset-reflecting measurable functions and on the category of smooth submersions. The latter, in particular, is a rich source of examples of linear layers. We recover dense and convolutional layers, as well as most complex structures arising in geometric deep learning. Indeed, a general approach to geometric deep learning, described in~\cite{monti2017geometric}, was an important inspiration for this work.

Describing a linear layer structure by means of smooth submersions between manifolds has unique advantages. We show that, in the case of convolutional layers, the smooth submersion determines the hyperparameters of the layer (such as stride or dilation). We envision that such smooth maps could be optimized (together with the regular parameters) during gradient descent. In our view, this is a promising, efficient alternative to the nested optimization schemes for hyperparameters proposed in~\cite{bengio2000gradient,lorraine2020optimizing}.

Describing single linear layers represents only a small fraction of a successful deep learning framework. We have been exploring in~\cite{vertechi2020parametric,vertechi2022machines} possible formalizations of the notion of {\em global neural network architecture}. First, we developed a framework, based on category theory, where neural architectures could be formally defined and implemented. Then, borrowing tools from functional analysis, we discussed the necessary assumptions to allow for backpropagation. Those works lie at the basis of the proposed single-layer framework. We believe that, in the future, it will be valuable to combine these approaches to define global architectures by means of parametric spans.

\bibliographystyle{abbrv}
\bibliography{References}

\begin{thebibliography}{10}

\bibitem{atwood2015diffusion}
J.~Atwood and D.~Towsley.
\newblock Diffusion-convolutional neural networks.
\newblock {\em arXiv preprint arXiv:1511.02136}, 2015.

\bibitem{baydin2018automatic}
A.~G. Baydin, B.~A. Pearlmutter, A.~A. Radul, and J.~M. Siskind.
\newblock Automatic differentiation in machine learning: a survey.
\newblock {\em Journal of Marchine Learning Research}, 18:1--43, 2018.

\bibitem{bengio2000gradient}
Y.~Bengio.
\newblock Gradient-based optimization of hyperparameters.
\newblock {\em Neural computation}, 12(8):1889--1900, 2000.

\bibitem{berline2003heat}
N.~Berline, E.~Getzler, and M.~Vergne.
\newblock {\em Heat kernels and Dirac operators}.
\newblock Springer Science \& Business Media, 2003.

\bibitem{boscaini2016learning}
D.~Boscaini, J.~Masci, E.~Rodoi{\`a}, and M.~Bronstein.
\newblock Learning shape correspondence with anisotropic convolutional neural
  networks.
\newblock In {\em Proceedings of the 30th International Conference on Neural
  Information Processing Systems}, pages 3197--3205, 2016.

\bibitem{eilenberg1966generalization}
S.~Eilenberg and G.~M. Kelly.
\newblock A generalization of the functorial calculus.
\newblock {\em Journal of Algebra}, 3(3):366--375, 1966.

\bibitem{kipf2016semi}
T.~N. Kipf and M.~Welling.
\newblock Semi-supervised classification with graph convolutional networks.
\newblock {\em arXiv preprint arXiv:1609.02907}, 2016.

\bibitem{lecun1989backpropagation}
Y.~LeCun, B.~Boser, J.~S. Denker, D.~Henderson, R.~E. Howard, W.~Hubbard, and
  L.~D. Jackel.
\newblock Backpropagation applied to handwritten zip code recognition.
\newblock {\em Neural computation}, 1(4):541--551, 1989.

\bibitem{lorraine2020optimizing}
J.~Lorraine, P.~Vicol, and D.~Duvenaud.
\newblock Optimizing millions of hyperparameters by implicit differentiation.
\newblock In {\em International Conference on Artificial Intelligence and
  Statistics}, pages 1540--1552. PMLR, 2020.

\bibitem{mac2013categories}
S.~Mac~Lane.
\newblock {\em Categories for the working mathematician}, volume~5.
\newblock Springer Science \& Business Media, 2013.

\bibitem{masci2015geodesic}
J.~Masci, D.~Boscaini, M.~Bronstein, and P.~Vandergheynst.
\newblock Geodesic convolutional neural networks on riemannian manifolds.
\newblock In {\em Proceedings of the IEEE international conference on computer
  vision workshops}, pages 37--45, 2015.

\bibitem{monti2017geometric}
F.~Monti, D.~Boscaini, J.~Masci, E.~Rodola, J.~Svoboda, and M.~M. Bronstein.
\newblock Geometric deep learning on graphs and manifolds using mixture model
  cnns.
\newblock In {\em Proceedings of the IEEE Conference on Computer Vision and
  Pattern Recognition}, pages 5115--5124, 2017.

\bibitem{richter2020categories}
B.~Richter.
\newblock {\em From categories to homotopy theory}, volume 188.
\newblock Cambridge University Press, 2020.

\bibitem{rumelhart1986learning}
D.~E. Rumelhart, G.~E. Hinton, and R.~J. Williams.
\newblock Learning representations by back-propagating errors.
\newblock {\em Nature}, 323(6088):533--536, 1986.

\bibitem{spivak2019generalized}
D.~I. Spivak.
\newblock Generalized lens categories via functors {$\mathcal{C}^{\rm
  op}\to\mathsf{Cat}$}.
\newblock {\em arXiv preprint arXiv:1908.02202}, 2019.

\bibitem{vertechi2022machines}
P.~Vertechi and M.~G. Bergomi.
\newblock Machines of finite depth: towards a formalization of neural networks.
\newblock {\em arXiv preprint arXiv:2204.12786}, 2022.

\bibitem{vertechi2020parametric}
P.~Vertechi, P.~Frosini, and M.~G. Bergomi.
\newblock Parametric machines: a fresh approach to architecture search.
\newblock {\em arXiv preprint arXiv:2007.02777}, 2020.

\end{thebibliography}

\end{document}